\documentclass[11pt,reqno,twoside]{amsart}

\usepackage{amsmath}
\usepackage{amsfonts}
\usepackage{amssymb}

\usepackage{color,soul}
\usepackage{amsmath,amssymb,amsthm,mathrsfs}
\usepackage{amsmath,color,graphicx}
\usepackage{verbatim, tikz}
\usepackage{amscd}
\usepackage{verbatim} 
\usepackage{bbm}
\usepackage{amscd}
\usepackage{url}
\usepackage{environ}

\numberwithin{equation}{section}

\def\be{\begin{equation}}
\def\ee{\end{equation}}

\def\n{\hbox{\kern -1.6pt{\rm N}}}
\def\H{\hbox{\kern -1.6pt{\rm H}}}
\def\IP{\hbox{\rm I\kern -1.6pt{\rm P}}}
\def\IC{{\hbox{\rm C\kern-.58em{\raise.53ex\hbox{$\scriptscriptstyle|$}}
    \kern-.55em{\raise.53ex\hbox{$\scriptscriptstyle|$}} }}}
\def\IN{\hbox{I\kern-.2em\hbox{N}}}
\def\IR{\hbox{\rm I\kern-.2em\hbox{\rm R}}}
\def\ZZ{\hbox{{\rm Z}\kern-.33em{\rm Z}}}
\def\IT{\hbox{\rm T\kern-.38em{\raise.415ex\hbox{$\scriptstyle|$}} }}
\def\ang{\hbox{$<$\kern-.42em\hbox{\rm )} }}

\def\b{\mathbb{B}}

\def\r{\mathbb{R}}
\def\z{\mathbb{Z}}
\def\H{\mathbb{H}}

\def\n{\mathbb{N}}
\def\N{\mathbb{N}}

\newcommand{\df}{{\, \stackrel{\mathrm{def}}{=}\, }}
\newcommand{\dist}{\operatorname{dist}}

\def\SO{\operatorname{SO}}
\newcommand {\ignore}[1]  {}
\begin{document}

\title[Lattice points counting and shrinking targets]{An application of lattice points counting \\ to shrinking target problems}
\date{November 15, 2016\\ The first-named author was supported by NSF grants DMS-1101320 and DMS-1600814.}

\author{Dmitry Kleinbock}
\address{Brandeis University, Waltham MA
02454-9110 {\tt kleinboc@brandeis.edu}}
\author{Xi Zhao}
\address{Brandeis University, Waltham MA 02454-9110
{\tt zhaoxi89@brandeis.edu}}


\maketitle

\begin{abstract}
We apply lattice points counting results 
to solve a shrinking target problem in the setting of discrete time geodesic flows on hyperbolic manifolds of finite volume.
\end{abstract}

\newtheorem{Thm}{Theorem}[section]
\newtheorem{Cor}[Thm]{Corollary}
\newtheorem{Prop}[Thm]{Proposition}
\newtheorem{Lemma}[Thm]{Lemma}
\newtheorem{Q}[Thm]{Question}
\newtheorem{Ex}[Thm]{Exercise}
\newtheorem{Rmk}[Thm]{Remark}
\newtheorem{Example}[Thm]{Example}
\newtheorem{Conj}[Thm]{Conjecture}
\newtheorem{Def}[Thm]{Definition}
\newtheorem{sublemma}[Thm]{Sublemma}

\newcommand\eq[2]{\begin{equation}\label{eq:#1}{#2}\end{equation}}
\newcommand {\equ}[1]     {\eqref{eq:#1}}

\newcommand {\new}[1]   {\textcolor{blue}{#1}}
\newcommand {\attention}[1]   {\textcolor{red}{#1}}
\newcommand {\myquestion}[1]   {\textcolor{green}{#1}}



\section{Introduction}

Let $(X,\mu)$ be a probability space and $T: X\mapsto X$    a measure-preserving transformation.
For a sequence of measurable sets $B_n\subset X$, 
consider the set $$ \limsup_n
T^{-n}B_n=\cap_{m=1}^{\infty}\cup_{n=m}^{\infty}T^{-n}B_n$$ of points $x\in X$ such that $T^nx\in B_n$
for infinitely many $n\in \N$. The Borel-Cantelli Lemma implies that if
$\sum_{n=1}^\infty \mu(B_n)$ is finite, then $\mu(\limsup_n
T^{-n}B_n) = 0$. The (converse) divergence case requires additional assumptions on the sets $B_n$. The classical Borel-Cantelli Lemma would imply that the measure of $\limsup_n
T^{-n}B_n$ is full if the sets $T^{-n}B_n$ are pairwise independent, an assumption which is hard to establish for deterministic dynamical systems.

 In many cases however a milder version of independence can be verified, still implying the full measure of the limsup set. Such results are usually referred to as dynamical Borel-Cantelli Lemmas. 
In many applications the family of sets $\{B_n\}$ is nested, and thus can be viewed as a `shrinking target', hence the terminology `Shrinking Target Problems'.  For example, if $\{B_n\}$ are shrinking balls centered at a point $p\in X$, a dynamical Borel-Cantelli Lemma can be thought of as a quantitative way to express density of trajectories of a generic point of $X$ at this fixed point $p$. Starting from the work of Phillip \cite{PW},  there have been many  results of this flavor. For example Sullivan \cite{1982} proved a Borel-Cantelli type theorem for cusp neighborhoods in hyperbolic manifolds of finite volume (here $p = \infty$), and the first named author with Margulis \cite{GM}  extended the result of Sullivan to non-compact Riemannian symmetric spaces. See also \cite{Ath1, Dolg, Gal, GS, note} for more references, and \cite{Ath2} for a nice survey of the area.


One particular example of a shrinking target property can be found in a paper by  Maucourant \cite{FM}. He considered nested balls in hyperbolic manifolds (quotients of the $n$-dimensional hyperbolic space $\H^n$) of finite volume, and proved the following theorem: 

\begin{Thm}
\label{maucorantstheorem}
Let $V$ be a finite volume hyperbolic manifold of real dimension $n$, $T^1V$  the unit tangent bundle of $V $, $\pi: T^1V \to V$ the canonical projection, $(\phi_t)_{t \in \r}$ the geodesic flow on $T^1V$, $\mu$ the Liouville measure on $T^1V$, and $d$ the Riemannian distance on $V$. Let $(B_t)_{t\geq 0}$ be a decreasing family of closed balls in $V$ (with respect to the metric $d$)
of radius $(r_t){_t\geq 0}$. Then for $\mu$-almost every $v$ in $T^1V$, the set
$\{t\geq 0: \pi (\varphi^tv)\in B_t\}$  is  bounded provided \eq{int}{\int_0^{\infty} r_t^{n-1}dt} converges, and is unbounded if \equ{int}
diverges.
\end{Thm}

Note that Maucourant's theorem holds for the continuous-time geodesic flow on $T^1V$. Now suppose that one replaces the continuous family $(B_t)_{t\geq 0}$ by a sequence $(B_t)_{t\in \n}$, and instead of the continuous geodesic flow considers the $h$-step discrete geodesic flow $(\varphi^{ht})_{t \in \n}$ for fixed $h \in \r_+$.  
The goal of this work is to provide additional argument needed to prove the Borel-Cantelli property, assuming some restrictions on the sequence $(B_t)$.

One of the ingredients in Maucourant's proof is a counting result for the number of lattice points inside balls in $\H^n$.  
To address  a 
discrete time analogue of Theorem \ref{maucorantstheorem} we use more refined lattice point counting results, namely 
an error term estimate for the number of lattice points in large balls in $\H^n$. 

We use the following notation throughout the paper:   for two non-negative functions $f$ and $g$, the notation $f(x) \ll g(x)$ means $f(x) \le C g(x)$ where $C> 0$ is a constant independent of $x$.

Here is a special case of our main result:  

\begin{Thm} \label{discrete1}
Let  $V$ be as in Theorem \ref{maucorantstheorem}, and let $(B_t)_{t \in \N}$ be a decreasing family of closed balls in $V
$ 
centered at $p_0\in V$  of radius $r_t$. Fix $h > 0$ and let $(\phi^{ht})_{t \in \n}$ be the $h$-step discrete geodesic flow. 
Then for $\mu$-almost every $v\in T^1V$,  the set
\eq{times} {
\{t \in \n: \pi(\phi^{ht} v) \in B_t\}
}
is finite provided the sum
\eq{intdiscrete}{\sum_{t \in \mathbb{N}} r_t^n } converges. Also, if one  assumes that \equ{intdiscrete}  diverges and, in addition, that 
\be
 \frac{-\ln  r_{t}}{r_t} \ll t \text{ for large enough }t,
 \label{additional}
 \ee
then for $\mu$-almost every $v$ in $T^1V$,  the set \equ{times}
is infinite.
\end{Thm}

That is, in the terminology of \cite{CK}, the sequence $(B_t)$ is a Borel-Cantelli sequence. Note that the difference in exponents in \equ{int} and \equ{intdiscrete} is due to the fact that Theorem \ref{maucorantstheorem}, unlike Theorem \ref{discrete1}, deals with a continuous time setting.

 It is well known that the geodesic flow on $T^1V$ as above has exponential decay of correlations, see e.g.\ \cite{CM, MR}. For systems with exponential mixing similar dynamical Borel-Cantelli Lemmas have been established before. For example, it follows from 
 \cite[Theorem 4.1]{note} that the set \equ{times}
will be  infinite provided \eq{theircondition}{\mu(B_t) \gg \frac{\ln t}{t}.}
 or,  equivalently, $r_t \gg \left(\frac{\ln t}{t}\right)^{1/n}$. This shows that  the restriction (\ref{additional}) is weaker than the one coming from  \cite[Theorem 4.1]{note}.
For example, take 
$r_t=\frac{C}{t^{\alpha}}$, where $\alpha\leq \frac{1}{n}$. 
Then \equ{intdiscrete}  diverges, and one can write $$\frac{-\ln  r_{t}}{r_t} =\frac{-\ln  \frac{C}{t^{\alpha}}}{\frac{C}{t^{\alpha}}}=\frac{1}{C}(- \ln C+\alpha \ln t )t^{\alpha} \ll t$$
when $t$ is large enough, therefore (\ref{additional}) is satisfied. Note that in the `critical exponent' case $\alpha = 1/n$ condition \equ{theircondition} fails to hold, thus the methods of  \cite{note} are not powerful enough to treat this case. 
The same also works for 
$$r_t=\frac{C}{t^{1/n}(\ln t)^{\beta}}$$ where 
$0 < \beta \le 1$: one has $$\frac{-\ln  r_{t}}{r_t} =\frac{-\ln  \frac{C}{t^{1/n}(\ln t)^{\beta}}}{\frac{C}{t^{1/n}(\ln t)^{\beta}}} =  \frac{1}{C} t^{1/n}(\ln t)^{\beta} \left(\frac1n \ln t+ \beta \ln (\ln t) - \ln C\right) \ll t$$
for large enough $t$.
\medskip

 We derive Theorem \ref{discrete1} from a more general statement, Theorem \ref{discrete2}, which  involves a technical  condition \equ{generalcondition} weaker than (\ref{additional}):


\begin{Thm} \label{discrete2}
Let $V$ be as above, and let $(B_t)_{t \in \N}$ and $h$ be as in Theorem \ref{discrete1}.
Then for $\mu$-almost every $v\in T^1V$,  the set \equ{times}
is finite provided the sum \equ{intdiscrete} converges. Also 
there exist $C_1,C_2 > 0$ 
such that 
if  \equ{intdiscrete}  diverges and, in addition, that \eq{generalcondition}{
\sum^{s}_{t=[s+ C_1\ln  r_{s}-C_2
]}r_t ^{n-1}\ll \sum_{t=1}^{s}r_t^{n} \quad\text{when $s$ is large enough},}
then for $\mu$-almost every $v\in  T^1V$,  the set \equ{times}
is infinite.
\end{Thm}

In the next section we will reduce Theorem \ref{discrete2} to a certain $L^2$ bound,   Theorem \ref{dis lemma6}, which will be verified in \S\ref{proofs}, and in \S\ref{completion} we will deduce Theorem  \ref{discrete1} from Theorem \ref{discrete2}.


\section*{Acknowledgments} The authors want to thank Dubi Kelmer, Keith Merrill, Amos Nevo, Hee Oh and the anonymous referee for useful comments.

\section{Reduction to Theorem \ref{dis lemma6}}\label{overview}

First note that for the divergence case of  Theorem \ref{discrete2} without loss of generality one can assume that $r_t  \to 0$ when  $t \to\infty$: indeed, if $(r_t)$ is bounded from below by a positive constant, then 
 the ergodicity of the geodesic flow implies  that
\eq{infoften}{
\pi(\varphi^{ht} v) \in B_t\text{  infinitely often for $\mu$-a.e. } v \in T^1V.}
Furthermore, for a fixed $R > 0$ we can  assume that 
 $r_t \le R$  for  all $t \in \n$. Indeed, if the theorem is proved under that assumption, then applying it to the family $\{B_t : t \ge t_0\}$ where $t_0$ is such that $r_t\le R$ when $t\ge t_0$, we still recover condition \equ{infoften}.
This $R$ will be fixed later, see (\ref{defr}).

Our proof  follows Maucourant's approach in \cite{FM}. 
Let us first introduce some terminology. Let $\mathcal{F}=(f_t)_{t \in\n}$ be a family of measurable functions on a probability space $(X,\mu)$. We call $\mathcal{F}$  decreasing   if  $f_s(x) \le f_t(x)$ for any $x\in X$ whenever 
 $s \ge t$.
Also let us write $$S_T[\mathcal{F}](x)\df \sum_{t=1}^T f_t(\phi^{ht}x),\quad I_T[\mathcal{F}]\df \sum_{t=1}^T\int_X f_t \,d\mu.$$ 


We are going to use the following proposition from Maucourant's paper:

\begin{Prop} \label{discrete prop}{\rm (\cite[Proposition 1]{FM})}
Let $\mathcal{F}=(f_t)_{t\in\N}$ be a decreasing family of non-negative measurable functions on $X$ such that $f_t\in L^2(X,\mu )$ for all $t$. Assume that $\lim\limits_{T \to \infty} I_T[\mathcal{F}]=\infty$, and that $S_T[\mathcal{F}]/I_T[\mathcal{F}]$ is bounded in $L^2$-norm as $T\to\infty$. Then, as $T\to\infty$, $S_T[\mathcal{F}]/I_T[\mathcal{F}]$ converges  to  $1$ weakly in $L^2(X,\mu )$, and 
for $\mu$-almost every $x$ in $X$ one has
\eq{limsup}{
\limsup_{T \to + \infty} \frac {S_T[\mathcal{F}](x)}{I_T[\mathcal{F}]}\ge 1.
}
\end{Prop}

We note that the above proposition was stated in \cite{FM} for the case of a continuous family of functions, but it is immediate to deduce a discrete version. 
To prove Theorem \ref{discrete2}, we 
will apply Proposition \ref{discrete prop} 
to the family of characteristic functions of $B_t$, i.e.\ take
\eq{defft}{
f_t(v)=\left\{
  \begin{array}{ll}
    1  & \hbox{if $d(\pi(
     v),p_0) \le r_t$}  \\
    0  & \hbox{otherwise.}
  \end{array}
\right.
}
It is decreasing because the family of balls $B_t$ is nested, and clearly 
$I_T[\mathcal{F}]$ is equivalent, up to a multiplicative constant, to $\sum_{t=1}^T r_t^n$. 
Also it is clear that the conclusion \equ{limsup} of Proposition \ref{discrete prop}  implies that the set \equ{times}
is infinite. Since the convergence case of  Theorems  \ref{discrete1} and \ref{discrete2} immediately follows from the Borel-Cantelli Lemma, we can see that 
 Theorem \ref{discrete2} can be reduced to proving a uniform $L^2$ bound for $S_T[\mathcal{F}]/I_T[\mathcal{F}]$, which is the subject of the following theorem:


\begin{Thm} \label{dis lemma6} Let $\mathcal{F}=(f_t)_{t\in\N}$ be as in \equ{defft}. Then 
there exist $C_1,C_2 > 0$ such that
if $I_T[\mathcal{F}]$ diverges when $T$ goes to $\infty$ and 
condition \equ{generalcondition} holds,
then the $L^2$-norm of $S_T[\mathcal{F}]/I_T[\mathcal{F}]$ is bounded for all $T \ge 1$.
 \end{Thm} 

\section{Proof of Theorem \ref{dis lemma6}}\label{proofs}

To prove Theorem \ref{dis lemma6}, following the same methodology as in \cite{FM}, we will apply a result on counting lattice points stated below (Theorem \ref{dis lemma5}) together with a measure estimate for the space of discrete geodesics (Theorem \ref{dis lemma4}).

\subsection{Counting lattice points}\label{counting}
Write $T^1V = \Gamma\backslash G$, where $\Gamma$ is a lattice in $G = \SO(n,1)$, the isometry group of $\widetilde V = \H^n$. Choose a lift $\tilde p_0\in\H^n$ of $p_0$ and for $r > 0$ 
and $i\in \n$, let us denote 
\[
\hat{\Gamma}_i(r)\df \big\{\gamma \in \Gamma : d(\tilde{p_0},\gamma \tilde{p_0}) \in(hi-r,hi+r]\big\}.
\]

Then  $$\# \hat{\Gamma}_i (r)  =\#(\Gamma \cap D_{hi+r} ) - \#( \Gamma \cap D_{hi-r}), $$ where $$D_t=\{g \in G:  d(g   \tilde{p_0}, \tilde{p_0}) \leq t\}.$$

An estimate for $\#\hat{\Gamma}_i (r)$ would follow from a reasonable estimate for the error term in the asymptotics of the size of $\Gamma \cap D_t$ for large $t$. Such estimates are due to Huber \cite{huber} for $n=2$ and to Selberg for the general case, see \cite{LF}, and also \cite{BO, AA} for more recent results of this flavor. Denote by  $m_G$ the Haar measure on $G$ which locally projects onto $\mu$.
The following is a consequence of \cite[Theorem 1]{LF}:






\begin{Thm} \label{theorem7.1} There exist  constants $0 < q < 1$ and $ t_1, c_{{1}} > 0$ such that 
$$| \# (\Gamma \cap D_t )- m_G(D_t)| \leq 
 c_{{1}}m_G(D_t)^{q},$$
for all $t > t_1$. \end{Thm}


An important property of the family $\{D_t\}$ is so-called H\"older well-roundedness, see \cite{AA}. In particular 
the following is true:

\begin{Prop} \label{basic facts} There exist   $t_2, c_{{2}}, c_{{3}} > 0$ such that:
\begin{itemize}
\item[\rm (i)] For any $\varepsilon <1$ and $t>t_2$,  we have that
\eq{holder}{
m_G(D_{t+\varepsilon})-m_G(D_{t-\varepsilon}) \leq c_{{2}} \varepsilon  m_G(D_{t-\varepsilon}).}
\item[\rm (ii)] For any $t>0$,
 \eq{mbond}{ m_G(D_{t}) \leq c_{{3}} e^{(n-1)t}.}
\end{itemize}
\end{Prop}


From the two statements above one can easily derive the following estimate:
\begin{Thm} \label{dis lemma5}
There exist constants $c_{{4}},c_5$ with the following property:  
if $0 < r  <1$ and $i\in\n$ are such that $$hi \geq \max(-c_{{4}} \ln r , r+ 
t_0), $$ 
where $t_0 = \max(t_1,t_2)$,
then 
\[
\#  \hat{\Gamma}_i(r)   \le c_5 r e^{(n-1) hi}.
\]
\end{Thm}

\begin{proof}[Proof of Theorem \ref{dis lemma5}]
%
Applying Theorem \ref{theorem7.1} for all $i$ with $hi-r>t_0$, we get that 
\begin{align*}
\# (\Gamma \cap D_{hi+r}) \le m_G(D_{hi+r})+  c_{{1}} m_G(D_{hi+r})^{q}
\end{align*}
and
\begin{align*}
\#(\Gamma \cap D_{hi-r})>m_G( D_{hi-r})-  c_{{1}} m_G(D_{hi-r})^{q}.
\end{align*}
Therefore, by \equ{holder} and \equ{mbond},  
we have: 
\begin{eqnarray*}
&& \#\{\gamma \in \Gamma : d(\tilde{p_0},\gamma \tilde{p_0}) \in [hi-r,hi+r]\}\\
 &=& \#( \Gamma \cap D_{hi+r}) -\#( \Gamma \cap D_{hi-r}) 
\\ &\leq&  m_G(D_{hi+r})+  c_{{1}} m_G(D_{hi+r})^{q}-m_G( D_{hi-r})+ c_{{1}} m_G(D_{hi-r})^{q}
\\ &\leq&  m_G(D_{hi+r})-m_G( D_{hi-r})+  c_{{1}}\big(m_G(D_{hi+r})^{q}+m_G(D_{hi-r})^{q}\big)
\\ &\ll &  r  m_G( D_{hi-r}) + \big(m_G(D_{hi+r})^{q}+m_G(D_{hi-r})^{q}\big)
\\ &\ll &  r e^{(n-1) (hi-r)} +e^{(n-1) {q}(hi+r)}+e^{(n-1) {q}(hi-r)}
\\ & \le &  r e^{(n-1) hi)} \left(1 + \frac{e^{-(n-1) {(1-q)}hi}}r\left(e^{(n-1)  {q} r}+e^{-(n-1)  {q} r}\right)\right) 
.
\end{eqnarray*}
\ignore{
 \begin{eqnarray*}
&& \#\{\gamma \in \Gamma| d(\tilde{p_0},\gamma \tilde{p_0}) \in [hi-\max\{r,r'\},hi+\max\{r,r'\}]\}\\
 &=& \#( \Gamma \cap D_{hi+\max\{r,r'\}}) -\#( \Gamma \cap D_{hi-\max\{r,r'\}}) 
\\ &\leq& e^{(n-1) (hi+r)}+t_0e^{(n-1) {q}(hi+r)}-e^{(n-1)(hi-r)}+t_0e^{(n-1) {q}(hi-r)}
\\ &\leq&  e^{(n-1) hi}(e^{(n-1) r}-e^{-(n-1) r}) +t_0e^{(n-1) {q}(hi+r)}+t_0e^{(n-1) {q}(hi-r)}
\end{eqnarray*}
}
Since $q < 1$ and $r<1$, we have $$e^{(n-1)  {q} r}+e^{-(n-1)  {q} r} < 2e^{n-1},$$ 
and clearly $\frac1r{e^{-(n-1) {(1-q)}hi}}\le  1$ 
whenever $hi \geq - \frac{\ln r}{(1-q)(n-1)}\,$. 
Summarizing the above, if $$hi \geq \max\left(-\frac{1}{(1-q)(n-1)} \ln r, r+ t_0\right) ,$$ 
then $ \# \hat{\Gamma}_i(r)   \ll r e^{(n-1) hi}$.
\end{proof}

\subsection{ The space of discrete geodesics on $\H^n$}
\label{space geodesic}
In this section  we will state  measure estimates for  spaces of geodesics on $\H^n$.

\begin{Def}\label{spacegeodesic}
We will write $\mathscr{G}$ as the space of oriented, unpointed continuous geodesics on $\H^n$. 
Using the fact that $T^1 \H^n$ can be written as $\mathscr{G} \times \r$, we can 
define a measure $\nu$ on $\mathscr{G} $ by $\tilde{\mu}=\nu \times dt$, where $\tilde{\mu}$ is the Liouville measure on $T^1\H^n$. 
\end{Def}

Then we will describe a similar definition for discrete geodesic flows. Namely:
\begin{Def}\label{spacedisgeodesic}
For fixed $h > 0$, $\mathscr{G}_h$ is the space of all $h$-step discrete geodesic trajectories: $\{\varphi^{ht}: t \in \z\}$. That is $ \mathscr{G}_h= \mathscr{G} \times S_h$ where $S_h$ is $[0,h]$ with $0$ and $h$ identified. In addition, since we can write $T^1\H^n = \mathscr{G}_h \times \mathbb{Z}h$, then we can define the measure $m$ on $\mathscr{G}_h$ by $m= \nu \otimes \lambda$, where $\nu$ is the measure on $\mathscr{G}$ defined above and $\lambda$ the Lebesgue measure on $S_h$.  Furthermore, the measure $\tilde{\mu}$ on the unit tangent bundle $T^1\H^n$ becomes the product of the measure $m$ on 
$\mathscr{G}_h$ with the 
counting measure on $ \mathbb{Z}h$.
\end{Def}

In \cite{FM}, Maucourant considered the space of continuous geodesics,
and estimated
the probability that a random geodesic visits two fixed balls in $V$ 
as follows:

\begin{Thm} \label{lemma4old} {\rm \cite[Lemma 4]{FM}}
There exists a constant $c_{{6}}>0$ such that, for any two balls in $\H^n$ of respective centers and radii $(o_1,r_1),(o_2,r_2)$ that satisfy $r_1$, $r_2<1$, and $d(o_1,o_2)>2$, the $\nu$-measure of continuous geodesics 
meeting those two balls is less than
\[
c_{{6}} r_1^{n-1}r_2^{n-1}e^{-(n-1) d(o_1,o_2)}.
\]
\end{Thm}


Here is a similar estimate for discrete geodesics on $T^1\H^n$:

\begin{Thm} \label{dis lemma4}
Consider two balls in $\H^n$ with respective centers and radii $(o_1,r_1)$, $(o_2,r_2)$ that satisfy $r_1<1$, $r_2<1$, and $d =d(o_1,o_2)>2$.  Also assume that $h>2 \min(r_1,r_2)$. Then 
the $m$-measure of the $h$-step geodesics which 
intersect those two balls is less than
$$
\begin{cases}
       2c_{{6}} r_1^{n-1}r_2^{n-1}e^{-(n-1) d}\min(r_1,r_2) & \text{if } \dist(d , h \z ) \le 2 \max(r_1,r_2) \\
       0 & \text{otherwise,}
     \end{cases}
$$
where $c_{{6}}$ is as in Theorem  \ref{lemma4old}.
\end{Thm}

\begin{proof} 
An $h$-step geodesic will fail to 
intersect both 
balls if for any $k$ we have
\begin{equation}\label{eq4}
| d
- k h | > 2\max(r_1,r_2);%
\end{equation}
in this case the measure we are to estimate is zero. So only if there is an integer $k$  such that \eqref{eq4} fails, 
can the $h$-step geodesic meet those balls. Using Theorem \ref{lemma4old} and the fact that the space of discrete geodesics is $\mathscr{G} \times S_h$ with measure $m= \nu \otimes dh$, one can notice that 
the measure of such geodesics is bounded by $2c_{{6}}\min(r_1,r_2)  r_1^{n-1}r_2^{n-1}e^{-(n-1) d}$. 
\end{proof}

\subsection{A bound for the $L^2$-norm of $S_T[\mathcal{F}]$}


Recall that for  $t \in \n$  we defined $f_t$ to be the characteristic function of $B_t$, which is a ball centered at $p_0\in V = \Gamma\backslash \H^n$  of radius $r_t$, see \equ{defft}, and considered 
the family of functions $\mathcal{F}=(f_t)_{t \in \n} $ on $T^1V$. Also we have chosen a lift $\tilde p_0\in\H^n$ of $p_0$. Now  define $\tilde{B}_t$ to be  a ball in $ \H^n$  centered at $\tilde p_0$ of radius $r_t$, and let $g_t$ be the characteristic function of $\tilde{B}_t$, 
Thus, the lift $\tilde{f}_t$ of $f_t$ to $T^1\tilde{V}$ satisfies
\[
\tilde{f}_t=\sum_{\gamma \in \Gamma} g_t \circ \gamma .
\]

Fix a fundamental domain $D$ of $ \H^n$ for $\Gamma$ containing $\tilde{p_0}$. 
 and define $$i_V(\tilde{p_0})\df \sup\{r\in \r:  B(\tilde{p_0}, r)\subset D\}.$$ Also  define
 \be R\df\min\big(i_V(\tilde{p_0})/4,1,h\big),\label{defr}\ee
 and, for $i\in\z_+$,
 $$\Gamma_i \df \left\{\gamma \in \Gamma : d(\tilde p_0,\gamma \tilde p_0) \in\left[hi-\frac{h}{2},hi+\frac{h}{2}\right)\right\}.$$

\ignore{Note that for the divergence case of  Theorem \ref{discrete2} without loss of generality we can assume that $r_t  \to 0$ when  $t \to\infty$: indeed, if $(r_t)$ is bounded from below by a positive constant, then 
 the ergodicity of the geodesic flow implies the condition (\ref{infoften}). Furthermore, we can  assume that $r_t \le R$  for  all $t \in \z_{\geq 0}$. Indeed, if the theorem is proved under that assumption, then applying it to the family $\{B_t : t \ge t_0\}$ where $t_0$ is such that $r_t\le R$ when $t\ge t_0$, we still recover the condition (\ref{infoften}).}

\begin{Thm} \label{lemma 7}
Let $D\subset \H^n$ be a 
fundamental domain   for $\Gamma$  such that $D$ contains the ball of center $\tilde{p_0}$ and of radius $3R$.
Then 
for all $T \in\N$,
\begin{eqnarray*}
\int\limits_{T^1V}S_T[\mathcal{F}](v)^2\,d\mu(v)
&\leq& 2\sum_{s=1}^T  \sum_{i=1}^{[s+\frac{6R}{h}]} \sum_{\gamma \in \Gamma_i} \int_{T^1D}\sum_{t=1}^s   g_s(v)g_t(\gamma\phi^{h(s-t)}v)\,d\tilde{\mu}(v).
\end{eqnarray*}
\end{Thm}

\begin{proof}

For  fixed $T\in\N$ and $v  \in T^1 V$,  we know that
\begin{eqnarray*}
S_T[\mathcal{F}](v)^2 = \left(\sum_{s=1}^{T}f_t(\phi^{ht} v)\right)\left( \sum_{t=1}^{T}f_s(\phi^{hs} v )\right) = 2 \sum_{s=1}^{T} \sum_{t \le s}f_t(\phi^{ht} v)f_s(\phi^{hs} v ).
\end{eqnarray*}

Now we can integrate $S_T[\mathcal{F}](v)^2$ over $T^1V$ and make a change of variable $w=\phi^{hs} v$.  Since $\phi^{hs}$ preserves the measure,   we have the following:

\[
\int\limits_{T^1V}S_T[\mathcal{F}](v)^2\,d\mu(v) \le 2\sum_{s=1}^{T} \sum_{t \le s}\int \limits_{T^1V}f_s (w)f_t (\phi^{h(t-s)} w )\,d\mu(w).
\]

By the fact that $\tilde{f}_t$ is the lift of $f_t$, we obtain that

\[
\int\limits_{T^1 V}S_T[\mathcal{F}](v)^2\,d\mu(v) \le 2\sum_{s=1}^{T} \sum_{t \le s}\int \limits_{T^1 \H^n}{\tilde{f}}_s (w){\tilde{f}}_t (\phi^{h(t-s)} w )\,d \tilde{\mu}(w).
\]

Since $\tilde{f}_t=\sum_{\gamma \in \Gamma}g_t\circ \gamma$, we can write
\[
\int\limits_{T^1 V}S_T[\mathcal{F}](v)^2\,d\mu(v) \le 2\sum_{s=1}^{T} \sum_{t \le s}\int \limits_{T^1 \H^n}\left( \sum_{\gamma \in \Gamma} g_s( \gamma w) \right) \left(\sum_{\gamma \in \Gamma} g_t( \gamma \phi^{h(t-s)} w )\right)\,d \tilde{\mu}(w).
\]


Recall that $D$ is the fundamental domain of $ \H^n$ for $\Gamma$. This insures that for all $w$ in $T^1 D $, 
 in the sum $\sum\limits_{\gamma\in \Gamma } g_s(\gamma w)$, all terms but the one corresponding to $\gamma=id$ are zero. So we have

\[\int\limits_{T^1V}S_T[\mathcal{F}](v)^2\,d\mu(v)
\le 2\sum_{s=1}^{T} \sum_{\gamma \in \Gamma}  \sum_{t \le s} \int \limits_{T^1D} g_s (w)g_t (\gamma\phi^{h(t-s)} w )\,d\tilde{\mu}(w).
\]

Making another change of variables  $v=-w$, where $
-w$
means the point in
$T^1D$
with the same projection as $w$ and the
tangent vector pointing in the opposite direction,
we deduce that
\[
\int\limits_{T^1V}S_T[\mathcal{F}](v)^2\,d\mu(v)
\le 2\sum_{s=1}^{T}\int  \limits_{T^1D} \sum_{\gamma \in \Gamma}  \sum_{t=1}^{s}  g_s (v)g_t (\gamma\phi^{h(s-t)} v )\,d\tilde{\mu}(v).
\]

For fixed $v\in T^1 D$, we know that $g_s(v)g_t(\gamma\phi^{h(s-t)}v)$ is zero when $\pi(v)\notin \tilde{B}_s $ or $\pi(\phi^{h(s-t)}v) \notin \gamma^{-1}\tilde{B}_t$, which implies that $g_s(v)g_t(\gamma\phi^{h(s-t)}v)$ vanishes when $$ |h(s-t)-d(\tilde{p_0},\gamma^{-1}\tilde{p_0})| > 2r_t+2r_s.$$ 
Since we know that $ 2r_t+2r_s <4R$, we can conclude that $g_s(v)g_t(\gamma\phi^{h(s-t)}v)$ vanishes when $t$ is outside of the interval $$\left[s-\frac{d(\tilde{p_0},\gamma^{-1}\tilde{p_0})}{h}-\frac{4R}{h}, s-\frac{d(\tilde{p_0},\gamma^{-1}\tilde{p_0})}{h}+\frac{4R}{h}\right].$$  Therefore, for any $v\in T^1V$ and any $s\in \n$, \eq{fact1}{ \# \big\{1 \leq t \leq s:  g_s(v)g_t(\gamma\phi^{h(s-t)}v)=1\big\} \leq \frac{8R}{h}.}

 Furthermore, the integral is zero if $\left|d(\tilde{p_0},\gamma^{-1}\tilde{p_0})-hi\right| > 2R> r_t+r_s$ for all $i$. Hence  this integral vanishes when $|hi-h(s-t)|>6R$,  i.e.\ when $$hi-h(s-t)>6R \quad \text {or}\quad hi-h(s-t)<-6R.$$ In particular, we see that the quantity $g_s(v)g_t(\gamma\phi^{h(s-t)}v)$ is zero if $$hi>hs+6R>h(s-t)+6R,\quad \text{i.e. } i>s+\frac{6R}{h}.$$

By the above fact and the fact that the union of all $\Gamma_i$ is $\Gamma$, we have
\begin{eqnarray*}
 \int\limits_{T^1V}S_T[\mathcal{F}](v)^2\,d\mu(v)
&\le& 2 \sum_{s=1}^T\sum\limits_{i\geq0} \sum_{\gamma \in \Gamma_i} \sum_{t=1}^s\int_{T^1D}g_s(v)g_t(\gamma\phi^{h(s-t)} v) \,d\tilde{\mu}(v).
\\ &=& 2 \sum_{s=1}^T  \sum_{i=0}^{[s+\frac{6R}{h}]} \sum_{\gamma \in \Gamma_i} \int_{T^1D}\sum_{t=1}^s g_s(v)g_t(\gamma\phi^{h(s-t)}v)\,d\tilde{\mu}(v).
\end{eqnarray*}
\end{proof}

Now let us  define \[ c_R =\frac{6R+2}{h} 
\] and   split the  estimate of Theorem \ref{lemma 7} into two parts:
\eq{twoparts}{
\begin{aligned}
\int\limits_{T^1V}S_T[\mathcal{F}](v)^2\,d\mu(v)
&\le  2 \sum_{s=1}^T  \sum_{i=1}^{[c_R]} \sum_{\gamma \in \Gamma_i} \int_{T^1D}\sum_{t=1}^s g_s(v)g_t(\gamma\phi^{h(s-t)}v)\,d\tilde{\mu}(v)
\\ &+  2 \sum_{s=1}^T  \sum_{i=[c_R]}^{[s+\frac{6R}{h}]} \sum_{\gamma \in \Gamma_i} \int_{T^1D}\sum_{t=1}^s  g_s(v)g_t(\gamma\phi^{h(s-t)}v)\,d\tilde{\mu}(v).
\end{aligned}
}


\subsection{ A bound on the first part of \equ{twoparts}}

It is not hard to estimate the first part.

\begin{Thm} \label{lemma 8}
There is constant $c_{{7}}$, only depending on $R$ and $h$, such that  for all $T \in\N$
$$\sum_{s=1}^T \sum^{[c_R]}_{i=0}\sum_{\gamma \in \Gamma_i} \int_{T^1D}\sum_{t=1}^s g_s(v)g_t(\phi^{h(s-t)}v)\,d\tilde{\mu}(v) 
 \le  c_{{7}} \sum ^T_{t=1} r_t^n  .
$$
\end{Thm}

\begin{proof}

Observing that $\cup^{c_R}_{i=3} \Gamma_i $ is a finite set, we write $N$ as its cardinal. 
Moreover, 
using  \equ{fact1}, we get  that 
\eq{equation 2 dis}{
\sum_{t=1}^{s} g_s(v)g_t(\gamma\phi^{h(s-t)}v)\le \frac{8R}{h}.
}
In addition, we notice the following facts:
\begin{itemize}
\item if $g_s(v)=0$, then $g_s(v)g_t(\gamma\phi^{s-t}v)$ in the left side vanishes;
\item if $g_s(v)=1$, then $g_s(v)g_t(\gamma\phi^{s-t}v)$ is at most $1$.
\end{itemize}

Therefore,  \equ{equation 2 dis} is equivalent to the following:
\begin{equation*}\label{equation 2 dis new}
\sum_{t=1}^{s} g_s(v)g_t(\gamma\phi^{h(s-t)}v)\le \frac{8R}{h}g_s(v).
\end{equation*}

This allows us to write
\begin{eqnarray*}
\sum_{s=1}^T \sum^{[c_R]}_{i=0}\sum_{\gamma \in \Gamma_i} \int_{T^1D}\sum_{t=1}^s g_s(v)g_t( \gamma \phi^{h(s-t)}v)\,d\tilde{\mu}(v)
\le  \frac{8NR}{h}  \sum^T_{s=1} \int_{T^1D} g_s(v)\,d\tilde{\mu}(v).
\end{eqnarray*}

Since $\int_{T^1D} g_s(v)d\tilde{\mu}(v)$ is equivalent to $r_s^n$, up to a multiplicative constant, there exists some positive constant $c_{{7}}$, depending only on $R$ and $h$, such that
\begin{eqnarray*}
 &&\sum_{s=1}^T \sum^{[c_R]}_{i=0}\sum_{\gamma \in \Gamma_i} \int_{T^1D}\sum_{t=1}^s g_s(v)g_t( \gamma \phi^{h(s-t)}v)\,d\tilde{\mu}(v)\le c_{{7}} \sum^T_{t=1} r_t^n .
\end{eqnarray*}

\end{proof}

\subsection{ A bound on the second part of \equ{twoparts}}
\begin{Thm} \label{lemma 9}
There exist constants $c_{{8}}$ and $c_{{9}}$, 
only depending on $R$, such that 
\begin{eqnarray*}
&& \sum_{s=1}^T  \sum_{i=[c_R]}^{[s+\frac{6R}{h}]}  \sum_{\gamma \in \Gamma_i} \int_{T^1D} \sum_{t=1}^sg_s(v)g_t(\gamma\phi^{h(s-t)}v)\,d\tilde{\mu}(v)
\\ &\le& c_{{8}}   \sum_{s=1}^T   r_s^{n} \sum^{[s-10R-4]}_{t=[s+ \frac{c_{{4}}}h\ln  r_{s}-\frac{6R}{h}-2]}r_t ^{n-1}+ c_{{9}} \sum_{s=1}^T r_s^n   \sum^{s}_{t=1} r_t^n,
\end{eqnarray*}
where $c_{{4}}$ is as in Theorem \ref{dis lemma5}.
\end{Thm}

\begin{proof}

Let us fix  $s$ and produce an upper bound on $\int\limits_{T^1D}\sum_{t=1}^s g_s(v)g_t(\gamma\phi^{h(s-t)}v) \,d\tilde{\mu}(v)
$. This requires the following observations:
\begin{enumerate}
\item  \equ{equation 2 dis} tells us that $ \sum_{t=1}^{s} g_s(v)g_t(\gamma\phi^{h(s-t)}v)\le \frac{8R}{h}$  for any $s$ and $v$.
\item We know that $|d(\tilde{p_0},\gamma^{-1} \tilde{p_0})-hi|<2R$, i.e. $$hi-2R<d(\tilde{p_0},\gamma^{-1} \tilde{p_0})<hi+2R.$$ Therefore, $i \geq c_R\geq \frac{6R+2}{h}$ implies that 
$$d(\tilde{p_0},\gamma^{-1} \tilde{p_0})> hi-2R>6R+2-2R> 2.$$ 
Hence, we know that the distance between the centers of $B(\tilde{p_0},r_s)$ and $B(\gamma^{-1}\tilde{p_0},r_{t})$ is greater than 2. Thus by Theorem \ref{dis lemma4}, the measure $m$ of the set of discrete geodesics intersecting both $B(\tilde{p_0},r_s)$ and $B(\gamma^{-1}\tilde{p_0},r_{t})$ is bounded by $2c_{{6}} r_s^{n-1}r^{n-1}_{t}e^{-(n-1) (hi-1)} r_s$.
\item Moreover, $D$ contains the ball of center $\tilde{p_0}$ with radius $3R$. So we know that for fixed    $v$,  $\#\{z \in \z h :  g_t(\phi^{hz}v)>0 \}\leq \frac{3R}{h}$.
\item In addition,   notice that $g_s(v)g_t(\gamma\phi^{h(s-t)}v)$ is not zero only if $$|hi-h(s-t)|<6R.$$ This implies that 
\begin{equation*}\label{t}
s-i-\frac{6R}{h}<t<s-i+\frac{6R}{h}.
\end{equation*}
\end{enumerate}

Now since $(r_t)$ is decreasing, for all $i \ge c_R = \frac{6R+2}{h}$, we have that
$$\int\limits_{T^1D}\sum_{t=1}^s g_s(v)g_t(\gamma\phi^{h(s-t)}v)\, d\tilde{\mu}(v)
\ll r_s^{n} r^{n-1}_{[s-i-\frac{6R}{h}-1]}e^{-(n-1)hi}.$$ 
Therefore, for all 
$i \ge c_R
$, we obtain that
\begin{equation*}\label{equation 3}
 \sum\limits_{\gamma \in \Gamma_i}\, \int \limits_{T^1D} \sum_{t=1}^{s} g_s(v)g_t(\gamma\phi^{h(s-t)}v) \,d\tilde{\mu}(v)
\ll  N_i r_s^{n} r_{[s-i-\frac{6R}{h}-1]}^{n-1} e^{-(n-1)hi}  ,
\end{equation*}
where $N_i$ is the number of elements of $\Gamma_i$ such that the integrated function is not zero. Now we can consider the sum over all $s$ and $i\ge c_R$:


\begin{eqnarray*}
&& \sum_{s=1}^T  \sum_{i=[c_R]}^{[s+\frac{6R}{h}]}  \sum_{\gamma \in \Gamma_i} \int_{T^1D} \sum_{t=1}^sg_s(v)g_t(\gamma\phi^{h(s-t)}v)\,d\tilde{\mu}(v)
\\ &\ll&\sum_{s=1}^T  \sum_{i=[c_R]}^{[s+\frac{6R}{h}]} r_s^{n} r_{s-i-\frac{6R}{h}-1}^{n-1} e^{-(n-1)hi} N_i 
\end{eqnarray*}


Our goal now is to estimate $N_i$. Recall Theorem \ref{dis lemma5}, which allows us to estimate $\#\hat\Gamma_i(r)$  when $hi \geq \max(-c_{{4}} \ln r, r+ t_0) $  for some constants $c_{{4}},t_0>0$. We will take $r =  r_{s-i-\frac{6R}{h}-1}$. Indeed, since $(r_t)$ is decreasing and we have assumed that $r_t\leq R<1$, it follows that $$- \ln r_{s} \geq - \ln r_{s-i-\frac{6R}{h}-1} $$ and $r_{s-i-\frac{6R}{h}-1}<R$.

Now let us define \be \label{defvr} V_{s} \df \max\left(\frac{- c_{{4}} \ln r_{s}}{h} ,\frac{1+t_0}{h}\right)+c_R. \ee 

Then $i \geq V_{s} $ implies that $$hi \geq \max\big(-c_{{4}} \ln  r_{s-i-\frac{6R}{h}-1} , r_{s-i-\frac{6R}{h}-1}+ t_0\big). $$ Meanwhile, we also know that $g_s(v)g_t(\gamma\phi^{h(s-t)}v)$ is not zero only if $$d(\tilde{p_0},\gamma^{-1}\tilde{p_0})\in [ih-r_{s-i-\frac{6R}{h}-1},ih+r_{s-i-\frac{6R}{h}-1}],$$ where $i$ is such that  $\gamma^{-1} \in \Gamma_i$.  Therefore  
\begin{equation*}
\begin{split}N_i&\leq \#\{\gamma :  d(\tilde p_0,\gamma \tilde p_0) \in [hi- r_{s-i-\frac{6R}{h}-1}, hi+ r_{s-i-\frac{6R}{h}-1}]\} \\&  \le c_5 e^{(n-1) (hi+1)}r_{s-i-\frac{6R}{h}-1}\end{split}\end{equation*} when $i \geq V_s$.

By applying the fact that $(r_t)$ is decreasing,  we have the following:

\begin{eqnarray*}
&& \sum_{s=1}^T \sum^{[s+\frac{6R}{h}]}_{i=[V_{s} ]} r_s^{n} r_{s-i-\frac{6R}{h}-1}^{n-1} e^{-(n-1)(hi-1)} N_i 
\ll \sum^T_{s=1} r_s^n   \sum^{s}_{t=1} r_t^n  .
\end{eqnarray*}

When $i<V_{s} $, we will use the   counting lattice point estimate (Theorem~\ref{theorem7.1}) to conclude that $N_i\ll e^{(n-1)hi}$. Recalling the definition of $V_{s} $, see (\ref{defvr}), we know that the assumption $i< -\frac  {c_{{4}}}h\ln  r_{s}$ implies that $i<V_{s} $. 
Meanwhile, since $(r_t)$ is decreasing, we have that
\begin{eqnarray*} 
 \sum_{s=1}^T   \sum_{i=c_R}^{[V_{s} ] }r_s^{n} r_{s-i-\frac{6R}{h}-1}^{n-1} e^{-(n-1)(hi-1)} N_i
\ll  \sum_{s=1}^T  r_s^{n}\sum^{s}_{t=[s+\frac  {c_{{4}}}h\ln  r_{s}-\frac{6R}{h}-2]}r_t ^{n-1}.
\end{eqnarray*}

Putting it all together, we conclude that
\begin{equation*}\begin{aligned}
&\sum_{s=1}^T  \sum_{i=[V_{s} ]}^{[s+\frac{6R}{h}]} \sum\limits_{\gamma \in \Gamma_i} \int \limits_{T^1D} \sum_{t=1}^s g_s(v)g_t(\gamma\phi^{h(s-t)}v)  \,d\tilde{\mu}(v)
\\ \le c_{{8}}&  \sum_{s=1}^T  r_s^{n} \sum^{s}_{t=[s+ \frac{c_{{4}}}{h}\ln  r_{s}-\frac{6R}{h}-2]}r_t ^{n-1} + c_{{9}} \sum^T_{s=1} r_s^n   \sum^{s}_{t=1} r_t^n.
\end{aligned}\end{equation*}
\end{proof}

\subsection{Completion of the proof of Theorem \ref{dis lemma6}}
\begin{proof}

Recall that so far we have
$$
\int_{T^1V}S_T[\mathcal{F}](v)^2 \,d\mu(v) \le   c_{{7}} \sum_{s=1}^{T} r^{n}_s+ c_{{8}} \sum_{s=1}^T r_s^{n} \sum^{s}_{t=[s+\frac{c_{{4}}}{h}\ln  r_{s}-\frac{6R}{h}-2]}r_t ^{n-1}+c_{{9}} \sum^T_{s=1} r_s^n   \sum^{s}_{t=1} r_t^n.$$

Now 
let us take $C_1 = c_{{4}}/h$, $C_2 = 2 + 6R/h$, and let us assume \equ{generalcondition}, i.e.\ that there exist $C,s_0$ such that
$${\sum^{s}_{t=[s+\frac{c_{{4}}}{h}\ln  r_{s}-\frac{6R}{h}-2]}r_t ^{n-1}\le C \sum_{t=1}^{s}r_t^{n} \quad\text{when $s \ge s_0$}.}
$$
Then we can write
\begin{equation*}\begin{split}
\int_{T^1V}S_T[\mathcal{F}](v)^2 \,d\mu(v) &\le   c_{{7}} \sum_{s=1}^{T} r^{n}_s+  c_{{8}}\sum_{s=1}^{s_0-1}  r_s^{n} \sum^{s}_{t=[s+\frac{c_{{4}}}{h}\ln  r_{s}-\frac{6R}{h}-2]} r_t ^{n-1}\\&+C \cdot c_{{8}}\sum_{s=s_0}^T  r_s^{n}  \sum_{t=1}^{s}r_t^{n}+c_{{9}} \sum^T_{s=1} r_s^n   \sum^{s}_{t=1} r_t^n\\
&\le  c_{{7}} \sum_{s=1}^{T} r^{n}_s+ c_{13} \sum^T_{s=1} r_s^n   \sum^{s}_{t=1} r_t^n + c_{14}
\end{split}\end{equation*}
 Since $\int_{T^1V}I_T[\mathcal{F}]^2 d 
 {\mu}(v)$ is equivalent, up to a multiplicative constant, to $\sum\limits_{s=1}^{T}r_s^{n}\sum\limits_{t=1}^{s}r_t^{n}$, and with the assumption that $\int_{T^1V}I_T[\mathcal{F}]^2 d 
 {\mu}(v) \to \infty$, $T \to \infty$, one can easily conclude that $ \frac{S_T[\mathcal{F}]}{I_T[\mathcal{F}]}$ is bounded in $L ^2$ -norm.
\end{proof}

\section{Proof of Theorem \ref{discrete1}}
\label{completion}
Recall that we are given  a non-increasing sequence $r_t$ which tends to $0$ as $t \to \infty$ and such that $\sum\limits_0^{\infty}r_t^{n}= \infty$ and in addition satisfying (\ref{additional}), that is, for some $C_0 , s_0 > 0$ it holds that
\be  \frac{- \ln  r_{s}}{r_s} \le C_0 s\text{ when }s\ge s_0.\label{additionalnew} \ee
We need to show that this sequence satisfies condition \equ{generalcondition}.  This will be an easy consequence of the following lemma:
\begin{Lemma}\label{thelastlemma} Under the above assumptions, for any  $C_1, C_2 > 0$ there exist $C_3,T > 0$ such that 
\[\sum^{s}_{t=[s+ C_1\ln  r_{s}-C_2]}r_t ^{n-1}\leq C_3 \sum_{t=1}^{s}r_t^{n}\text{ when }s\geq T. 
\]
\end{Lemma}

\begin{proof}
By (\ref{additionalnew}), \be - C_1 \ln  r_{s}\le C_1 C_0 sr_s \text{ when }s\ge s_0.
\label{additionalc0}
\ee
Take $s_1$  such that $- \ln r_s\geq  \frac{ C_2 }{C_1}$  when $s>s_1$. This and (\ref{additionalc0}) imply that $$ C_2  \leq C_1 C_0 sr_{s}$$ \text{and} 
\be 2C_1 C_0 sr_{s} \geq -C_1 \ln  r_{s}+C_2. \label{srs0prime} \ee
Since $r_s$ is non-increasing, 
(\ref{srs0prime}) 
implies that
\begin{align*} 
2C_1C_0 sr_{[s+ C_1 \ln  r_{s}-C_2]} \geq -C_1 \ln  r_{s}+C_2
\end{align*} 
when $s>\max(s_0,  s_1)$.
Due to the fact that $0<r_s< r_{[s+ C_1 \ln  r_{s}-C_2]}<1$, we have that
\begin{align*} 
2C_1C_0  sr_{[s+ C_1 \ln  r_{s}-C_2]}
&\geq \frac{1}{2} (-C_1 \ln  r_{s}+C_2)(1 +r_{[s+ C_1 \ln  r_{s}-C_2]}-r_{s});
\end{align*} 
thus
\begin{align*} 
2C_1C_0  sr_{[s+ C_1 \ln  r_{s}-C_2]}
&\geq (-C_1 \ln  r_{s}+C_2)(1 +r_{[s+ C_1 \ln  r_{s}-C_2]}-r_{s})\\
&\geq (-C_1 \ln  r_{s}+C_2)(1-r_{s} )+ (-C_1 \ln  r_{s}+C_2)r_{[s+ C_1 \ln  r_{s}-C_2]}.
\end{align*} 
Therefore, when $s>\max(s_0, s_1)$, we obtain that
\be 
( 
2C_1C_0  s+ C_1 \ln  r_{s}-C_2)r_{[s+  C_1 \ln  r_{s}-C_2]}\geq (-C_1 \ln  r_{s}+C_2)(1-r_{s} ).\label{srs2}
\ee

Now take $s_2 > 0$  such that $ r_s<\frac1{4C_1^2C_0^2} $  when $s >  s_2$, and let 
  $T\df \max(s_0,  s_1,\ s_2)$.
Then (\ref{srs0prime}) implies that, when $s>T$,

\[
 s \geq 
2C_1 C_0( -C_1 \ln  r_{s}+C_2).
\]
Thus, by adding $2C_1C_0 s+( 2C_1C_0+1) (C_1\ln  r_{s}- C_2)$ 
to both sides, we conclude that, when $s>T$
$$ ( 2C_1C_0+1) ( s+C_1\ln  r_{s}-C_2)\geq 2C_1C_0 s+C_1\ln  r_{s}-C_2 .$$
Now let us define  $C_3 \df 2C_1C_0 +1$.

Then we have that, when $s>T$
$$  C_3 ( s+C_1\ln  r_{s}-C_2)\geq 2C_1C_0 s+C_1\ln  r_{s}-C_2 .$$

which, in view of (\ref{srs2}), implies 
$$C_3 ( s+ C_1 \ln  r_{s}-C_2)r_{[s+  C_1 \ln  r_{s}-C_2]}\geq (-C_1 \ln  r_{s}+C_2)(1-r_{s} ).$$

Since $r_{[s+  C_1 \ln  r_{s}-C_2]}>0$, the above inequality implies that, when $s\ge T$,
\be
C_3 ( s+ C_1 \ln  r_{s}-C_2)r^n_{[s+  C_1 \ln  r_{s}-C_2]}\geq (-C_1 \ln  r_{s}+C_2)(1-r_{s} )r^{n-1}_{[s+  C_1 \ln  r_{s}-C_2]}.\label{twosided}\ee

On the other hand, since $r_s$ is non-increasing, one will notice that
$$ C_3 ( s+C_1\ln  r_{s}-C_2)r_{[s+ C_1\ln  r_{s}-C_2)]}^{n}\leq  C_3 \sum_{t=1}^{[s+C_1\ln  r_{s}-C_2)]}r_t^{n},$$
and 
\begin{align*} 
\sum^{s}_{t=[s+ C_1\ln  r_{s}-C_2]}(r_t ^{n-1}-r_t ^{n})&=\sum^{s}_{t=[s+ C_1 \ln  r_{s}-C_2]}(1-r_t )r_t ^{n-1}\\
&\leq (- C_1 \ln  r_{s}+C_2)(1-r_{s}) r_{[s+C_1 \ln  r_{s}-C_2]} ^{n-1} .
\end{align*}
Therefore, by (\ref{twosided}), we have that, when $s\ge T$,
$$  \sum^{s}_{t=[s+ C_1 \ln  r_{s}-C_2]}(r_t ^{n-1}-r_t ^{n})\leq C_3\sum_{t=1}^{[s+ C_1 \ln  r_{s}-C_2]}r_t^{n},$$
and hence
$$  \sum^{s}_{t=[s+ C_1 \ln  r_{s}-C_2]}r_t ^{n-1}\leq C_3 \sum_{t=1}^{[s+ C_1 \ln  r_{s}-C_2]}r_t^{n}+ \sum^{s}_{t=[s+ C_1 \ln  r_{s}-C_2]}r_t ^{n}\leq C_3 \sum_{t=1}^{s}r_t^{n}.$$

\end{proof}
This shows that (\ref{additional}) implies \equ{generalcondition},  and finishes the proof of   Theorem \ref{discrete1}. \qed

\begin{thebibliography}{99}







\bibitem{Ath2} J.\ Athreya, \textsl{Logarithm laws and shrinking target properties}, Proc.\ Math.\ Sci.\ {\bf 119} (2009), no.\ 4, 541--557.
\bibitem{Ath1} \bysame, \textsl{Cusp excursions on parameter spaces}, J.\ Lond.\ Math.\ Soc.\ {\bf 87} (2013) 741--765.

\bibitem{BO} Y.\ Benoist and H.\ Oh, \textsl{Effective equidistribution of $S$-integral points on symmetric varieties}, Annales de L'Institut Fourier {\bf 62} (2012), 1889--1942.

\bibitem{CK}
N.\ Chernov and D.\ Kleinbock, \textsl{Dynamical Borel-Cantelli lemmas for Gibbs measures}, Israel J.\ Math.\ {\bf 122} (2001), 1--27. 



\bibitem{Dolg} D.\ Dolgopyat, \textsl{Limit theorems for partially hyperbolic systems}, Trans.\ Amer.\ Math.\
Soc. {\bf 356} (2004), 1637--1689.
\bibitem{Gal} S.\ Galatolo,  \textsl{Dimension and hitting time in rapidly mixing systems},  Math.\ Res.\ Lett.\ {\bf 14} (2007), 797--805.

\bibitem{AA}
A.\ Gorodnik and A.\ Nevo, \textsl{Counting lattice points}, J.\ Reine Angew.\ Math.\ {\bf 663} (2012), 127--176. 


\bibitem{GS} A.\ Gorodnik and N.\ Shah,  \textsl{Khinchin's theorem for approximation by integral points on quadratic varieties},  Math.\ Ann.\ {\bf 350} (2011), 357--380.

\bibitem{note}
N.\ Haydn, M.\ Nicol, T.\ Persson   and S.\ Vaienti, \textsl{A note on Borel-Cantelli lemmas for non-uniformly hyperbolic dynamical systems}, Ergodic Theory Dynam.\ Systems {\bf 33} (2013), no.\ 2, 475--498. 

\bibitem{huber}H.\ Huber, \textsl{\"Uber eine neue Klasse automorpher Functionen und eine Gitterpunktproblem
in der hyperbolischen Ebene}, Comment.\ Math.\ Helv.\ {\bf 30} (1956), 20--62.


\bibitem{GM}
D.Y.\ Kleinbock and G.A.\ Margulis \textsl{Logarithm laws for flows on homogeneous spaces}, Invent.\ Math.\ {\bf 138} (1999), no.\ 3, 451--494. 


\bibitem{FM}
F.\ Maucourant, \textsl{Dynamical Borel-Cantelli lemma for hyperbolic spaces}, Israel J.\ Math.\ {\bf 152} (2006), no.\ 1, 143--155.

\bibitem{LF} P.\ Lax and R.\ Phillips,
\textsl{The asymptotic distribution of lattice points in Euclidean and Non-Euclidean spaces}, J.\ Funct.\ Anal.\
{\bf 46}
(1982), 280--350.

\bibitem{CM}
 C.C.\ Moore, \textsl{Exponential decay of correlation coefficients for geodesic flows}, in: Group representations, ergodic theory, operator algebras, and mathematical physics (Berkeley, CA, 1984), 163--181, 
Math.\ Sci.\ Res.\ Inst.\ Publ.\ {\bf 6}, Springer, New York, 1987. 
 

\bibitem{PW}
W.\ Philipp, \textsl{Some metrical theorems in number theory}, Pacific J.\ Math.\ {\bf 20} (1967), no.\ 1, 109--127.



\bibitem{MR}
 M.\ Ratner, \textsl{The rate of mixing for geodesic and horocycle flows}, Ergodic Theory Dynam.\ Systems {\bf 7} (1987), no.\ 2, 267--288. 





\bibitem{1982}
D.\ Sullivan, \textsl{Disjoint spheres, approximation by quadratic numbers and the logarithm law for geodesics}, Acta Math\. {\bf 149} (1982), 215--237.




\end{document}